\documentclass[11pt, oneside]{amsart}
 \usepackage[text={5.58in,8.5in},centering,letterpaper,dvips]{geometry}
\usepackage{graphicx}
\usepackage{amsfonts}
\usepackage{epsf}
\usepackage{amssymb}
\usepackage{amsmath}
\usepackage{amscd}
\usepackage{tikz}
\usepackage{pdfpages}
\usepackage{fancyhdr}
\usepackage{setspace}
\usepackage{hyperref}
\usepackage[all]{xy}
\usetikzlibrary{matrix}
\usepackage{verbatim}
\usepackage{enumerate}
\usepackage{amsbsy}

\theoremstyle{theorem}
\newtheorem{theorem}{Theorem}[section]
\newtheorem{proposition}[theorem]{Proposition}
\newtheorem{lemma}[theorem]{Lemma}
\newtheorem{question}[theorem]{Question}
\newtheorem{corollary}[theorem]{Corollary}

\newtheorem*{GPRC}{Generalized Property R Conjecture}
\newtheorem*{SGPRC}{Stable Generalized Property R Conjecture}

\theoremstyle{definition}

\newtheorem{examples}[theorem]{Examples}

\newcommand{\CP}{\mathbb{CP}}

\newcommand{\Tt}{\mathcal T}
\newcommand{\Ss}{\mathcal S}

\newcommand{\wh}[1]{\widehat{#1}}

\newcommand{\A}{\alpha}
\newcommand{\n}{\beta}
\newcommand{\g}{\gamma}
\newcommand{\pd}{\partial}

\newcommand{\emp}{\emptyset}
\newcommand{\X}{\times}
\newcommand{\Aa}{\alpha^+}
\newcommand{\Nn}{\beta^+}
\newcommand{\Gg}{\gamma^+}

\makeatletter
\def\@seccntformat#1{%
  \protect\textup{\protect\@secnumfont
    \ifnum\pdfstrcmp{subsection}{#1}=0 \bfseries\fi% subsection # in \bfseries
    \csname the#1\endcsname
    \protect\@secnumpunct
  }%
}  
\makeatother

\makeatletter
\newtheorem*{rep@theorem}{\rep@title}
\newcommand{\newreptheorem}[2]{%
\newenvironment{rep#1}[1]{%
 \def\rep@title{#2 \ref{##1}}%
 \begin{rep@theorem}}%
 {\end{rep@theorem}}}
\makeatother

\newreptheorem{theorem}{Theorem}
\newreptheorem{lemma}{Lemma}
\newreptheorem{question}{Question}
\newreptheorem{corollary}{Corollary}
\newreptheorem{proposition}{Proposition}
\newreptheorem{problem}{Problem}

\begin{document}

\rhead{\thepage}
\lhead{\author}
\thispagestyle{empty}

%\tableofcontents
%\listoffigures

\raggedbottom
\pagenumbering{arabic}
\setcounter{section}{0}

%%%%%%%%%%%%%%%%%%%%%%%%%%%%%%%%%%%%%%%%%%%%%%%%%%%%%%%%
%%%%%%%%%%%%%%%%%%%%%%%%%%%%%%%%%%%%%%%%%%%%%%%%%%%%%%%%
%%%%%%%%%%%%%%%%%%%%%%%%%%%%%%%%%%%%%%%%%%%%%%%%%%%%%%%%

\title{Characterizing Dehn surgeries on links via trisections}
\date{\today}

\author{Jeffrey Meier}
\address{Department of Mathematics, Indiana University, 
Bloomington, IN 47408}
\email{jlmeier@indiana.edu}
\urladdr{http://pages.iu.edu/~jlmeier} 

\author{Alexander Zupan}
\address{Department of Mathematics, University of Nebraska-Lincoln, Lincoln, NE 68588}
\email{zupan@unl.edu}
\urladdr{http://www.math.unl.edu/~azupan2}

\begin{abstract}
	We summarize and expand known connections between the study of Dehn surgery on links and the study of trisections of closed, smooth 4-manifolds.  In addition, we describe how the potential counterexamples to the Generalized Property R Conjecture given by Gompf, Scharlemann, and Thompson yield genus four trisections of the standard four-sphere that are unlikely to be standard.  Finally, we give an analog of the Casson-Gordon Rectangle Condition for trisections that can be used to obstruct reducibility of a given trisection.
\end{abstract}

\maketitle

%%%%%%%%%%%%%%%%%%%%%%%%%%%%%%%%%%%%%%%%%%%%%%%%%%%%%%%%
%%%%%%%%%%%%%%%%%%%%%%%%%%%%%%%%%%%%%%%%%%%%%%%%%%%%%%%%
\section{Outline}\label{sec:outline}
%%%%%%%%%%%%%%%%%%%%%%%%%%%%%%%%%%%%%%%%%%%%%%%%%%%%%%%%
%%%%%%%%%%%%%%%%%%%%%%%%%%%%%%%%%%%%%%%%%%%%%%%%%%%%%%%%

The purpose of this note is to use both new and existing results to make clear the significant role of the trisection theory of smooth 4-manifolds in the classification of Dehn surgeries on links.  The theory of Dehn surgery on knots has been thoroughly developed over the past forty years.  In general, this research has focused on two major questions:  First, which manifolds can be obtained by a surgery on a knot in a given manifold $Y$?  Second, given a pair of manifolds $Y$ and $Y'$, for which knots $K \subset Y$ does there exist a surgery to $Y'$?  These two questions have contributed to the growth of powerful tools in low-dimensional topology, such as sutured manifold theory, the notion of thin position, and Heegaard Floer homology.  For example, over the last 15 years, the Heegaard Floer homology theories of Ozsv\'ath and Szab\'o have dramatically deepened our collective understanding of Dehn surgeries on knots (see, for instance,~\cite{Ozsvath-Szabo_Lectures_2006}).

If we replace the word ``knot" with ``link" in the preceding paragraph, the situation changes significantly; for example, the classical Lickorish-Wallace Theorem asserts that every 3-manifold $Y$ can be obtained by surgery on a link in $S^3$~\cite{Lickorish_A-representation_1962,Wallace_Modifications_1960}. For the second general question, concerning which links in a given 3-manifold $Y$ yield a surgery to another given 3-manifold $Y'$, we observe the following basic fact:  Two framed links that are handleslide equivalent surger to the same 3-manifold~\cite{Kirby_A-calculus_1978}.  Thus, surgery classification of links is necessarily considered up to handleslide equivalence, and tools which rely on the topology of a knot exterior $S^3 \setminus \nu(K)$ are not nearly as useful, since handleslides can significantly alter this topology.

Understanding link surgeries in particular 3-manifolds is intimately connected to smooth 4-manifold topology.  Every smooth 4-manifold $X$ can be described by a handle decomposition, characterized by the attaching link $L$ for the 2-handles, which is contained in the boundary of the union of the 0- and 1-handles.  In other words, $X$ is associated to a framed link $L \subset  \#^k(S^1 \X S^2)$ such that Dehn surgery on $L$ yields $\#^{k'} (S^1 \X S^2)$.  Conversely, such a link $L$ (which we will call \emph{admissible}) describes a handle decomposition of a smooth 4-manifold, which we denote $X_L$.  Thus, classifying all such surgeries would be equivalent to classifying all smooth 4-manifolds.

Clearly, this is an insurmountable task, but to make the problem more tractable, we consider various restrictions placed on the parameters $k$, $k'$, and $n$, where $n$ represents the number of components of the link $L$.  For example, let $k = k' = 0$.  In the case that $n=1$, Gordon and Luecke proved that knots are determined by their complements, and thus the only knot in $S^3$ that admits an integral $S^3$ surgery (a \emph{cosmetic surgery}) is a $\pm 1$-framed unknot~\cite{Gordon-Luecke_Complements_1989}. In this paper, we will describe the proof of the following theorem from~\cite{Meier-Zupan_Genus-two_2017}.

\begin{theorem}\label{main1}
If $L \subset S^3$ is a two-component link with tunnel number one with an integral surgery to $S^3$, then $L$ is handleslide equivalent to a 0-framed Hopf link or $\pm 1$-framed unlink.
\end{theorem}

Another significant case occurs when $k = 0$ and $k' = n$.  In other words, we wish to understand $n$-component links in $S^3$ with surgeries to $\#^n (S^1 \X S^2)$.  We call such a link $L$ an \emph{R-link}, noting that R-links correspond precisely to the collection of geometrically simply-connected homotopy 4-spheres, i.e. homotopy 4-spheres built without 1-handles.  The Generalized Property R Conjecture (GPRC), Kirby Problem 1.82~\cite{Kirby_Problems_1978}, contends that every R-link is handleslide equivalent to a 0-framed unlink.  The conjecture is known to be true in the case $n=1$ via Gabai's proof of Property R~\cite{Gabai_FoliationsIII_1987}.  In~\cite{Meier-Schirmer-Zupan_Classification_2016}, the authors, in collaboration with Trent Schirmer, proved a stable version of the GPRC for a class of links.

\begin{theorem}\label{main2}
If $L \subset S^3$ is an $n$-component R-link with tunnel number $n$, then the disjoint union of $L$ with a 0-framed unknot is handleslide equivalent to a 0-framed unlink.
\end{theorem}

As foreshadowed above, the proofs of these theorems are 4-dimensional in nature, utilizing a prominent new tool: \emph{trisections} of smooth 4-manifolds.  A trisection is a decomposition of a 4-manifold $X$ into three simple pieces, a 4-dimensional version of a 3-dimensional Heegaard splitting.  Elegantly connecting the two theories, Gay and Kirby proved that every smooth 4-manifold admits a trisection, and every pair of trisections for a given 4-manifold have a common stabilization~\cite{Gay-Kirby_Trisecting_2016}, mirroring the Reidemeister-Singer Theorem~\cite{Reidemeister_Zur-dreidimensionalen_1933, Singer_Three-dimensional_1933} in dimension three.

Unlike Heegaard splittings, however, the stabilization operation of Gay and Kirby can be broken into three separate operations, called \emph{unbalanced stabilizations} of types 1, 2, and~3~\cite{Meier-Schirmer-Zupan_Classification_2016}.  A trisection is said to be \emph{standard} if it is an unbalanced stabilization of the genus zero trisection of $S^4$, and thus, every trisection of $S^4$ becomes standard after some number of Gay-Kirby stabilizations.
 
In Section~\ref{sec:R-stab}, we describe a process by which an R-link $L$ paired with an \emph{admissible} Heegaard surface $\Sigma$ for its exterior is converted to a trisection $\Tt(L,\Sigma)$ of the 4-manifold $X_L$.  The new main result of this paper is a technical theorem that connects R-links to properties of these trisections.  The terms \emph{$\{2\}$--standard and $\{2,3\}$--standard} refer to trisections that become standard after allowing restricted types of unbalanced stabilizations; we will postpone the rigorous definitions for now.

\begin{theorem}\label{thm:equiv}
Suppose $L$ is an R-link and $\Sigma$ is any admissible surface for $L$.
{\begin{enumerate}
\item If $L$ satisfies the GPRC, then $\Tt(L,\Sigma)$ is $\{2\}$--standard.
\item The link $L$ satisfies the Stable GPRC if and only if $\Tt(L,\Sigma)$ is $\{2,3\}$--standard.
\end{enumerate}}
\end{theorem}

In Section~\ref{sec:GST}, we analyze examples of Gompf-Schlarlemann-Thompson, the most prominent possible counterexamples to the GPRC.  The first step in a program to disprove the GPRC via Theorem~\ref{thm:equiv} is to find low-genus admissible surfaces for these links, along with diagrams for their induced trisections.  We outline this process; extensions of Section~\ref{sec:GST} will appear in forthcoming work~\cite{Meier-Zupan_Fibered_}.

In Section~\ref{sec:rect}, we introduce an analog of the Casson-Gordon Rectangle Condition~\cite{Casson-Gordon_Reducing_1987} for trisection diagrams, giving a sufficient condition for a trisection diagram to correspond to an irreducible trisection.

\subsection*{Acknowledgements}\ 
The first author is supported by NSF grants DMS-1400543 and DMS-1664540, and the second author is supported by NSF grant DMS-1664578 and NSF-EPSCoR grant OIA-1557417.

%%%%%%%%%%%%%%%%%%%%%%%%%%%%%%%%%%%%%%%%%%%%%%%%%%%%%%%%
%%%%%%%%%%%%%%%%%%%%%%%%%%%%%%%%%%%%%%%%%%%%%%%%%%%%%%%%
\section{Trisections and admissible links}\label{sec:trisections}
%%%%%%%%%%%%%%%%%%%%%%%%%%%%%%%%%%%%%%%%%%%%%%%%%%%%%%%%
%%%%%%%%%%%%%%%%%%%%%%%%%%%%%%%%%%%%%%%%%%%%%%%%%%%%%%%%

All manifolds are connected and orientable, unless otherwise stated.  We will let $\nu( \cdot )$ refer to an open regular neighborhood in an ambient manifold that should be clear from context.  The \emph{tunnel number} of a link $L \subset Y$ is the cardinality of the smallest collection of arcs $a$ with the property that $Y \setminus \nu(L \cup a)$ is a handlebody.  In this case, $\pd \nu(L \cup a)$ is a \emph{Heegaard surface} cutting $Y \setminus \nu(L)$ into a handlebody and a compression body.  A \emph{framed} link refers to a link with an integer framing on each component.

Let $L$ be a framed link in a 3-manifold $Y$, and let $a$ be a framed arc connecting two distinct components of $L$, call them $L_1$ and $L_2$.  The framings of $L_1$, $L_2$ and $a$ induce an embedded surface $S \subset Y$, homeomorphic to a pair of pants, such that $L_1 \cup L_2 \cup a$ is a core of $S$.  Note that $S$ has three boundary components, two of which are isotopic to $L_1$ and $L_2$.  Let $L_3$ denote the third boundary component, with framing induced by $S$.  If $L'$ is the framed link $(L \setminus L_1) \cup L_3$, we say that $L'$ is obtained from $L$ by a \emph{handleslide} of $L_1$ over $L_2$ along $a$.

If two links are related by a finite sequence of handleslides, we say they are \emph{handleslide equivalent}.  It is well-known that Dehn surgeries on handleslide equivalent framed links yield homeomorphic 3-manifolds~\cite{Kirby_A-calculus_1978}.  Recall that an R-link is an $n$-component link in $S^3$ with a Dehn surgery to the manifold $\#^n(S^1 \X S^2)$, which we henceforth denote by $Y_n$.  Let $U_n$ denote the $n$-component zero-framed unlink in $S^3$.  If an R-link $L$ is handleslide equivalent to $U_n$, we say that $L$ has \emph{Property R}.  If the split union $L \sqcup U_r$ is handleslide equivalent to $U_m$ for some integers $r$ and $m$, we say that $L$ has \emph{Stable Property R}.  The following conjectures are well-known; the first is Kirby Problem 1.82~\cite{Kirby_Problems_1978}.

\begin{GPRC}
Every R-link has Property R.
\end{GPRC}

\begin{SGPRC}
Every R-link has Stable Property R.
\end{SGPRC}

In this section, we explore the relationship between R-links (and a more general family we call admissible links) and trisections of the smooth 4-manifolds that can be constructed from these links.

Let $X$ be a smooth, orientable, closed 4-manifold.  A \emph{$(g;k_1,k_2,k_3)$--trisection} $\Tt$ of $X$ is a decomposition $X = X_1\cup X_2\cup X_3$ such that
	\begin{enumerate}
		\item Each $X_i$ is a four-dimensional 1--handlebody, $\natural^{k_i}(S^1\times B^3)$;
		\item If $i\not=j$, then $H_{ij} = X_i\cap X_j$ is a three-dimensional handlebody, $\natural^g(S^1\times D^2)$; and
		\item The common intersection $\Sigma = X_1\cap X_2\cap X_3$ is a closed genus $g$ surface.
	\end{enumerate}
	The surface $\Sigma$ is called the \emph{trisection surface}, and the parameter $g$ is called the \emph{genus} of the trisection.    The trisection $\Tt$ is called \emph{balanced} if $k_1=k_2=k_3=k$, in which case it is called a \emph{$(g;k)$--trisection}; otherwise, it is called \emph{unbalanced}.  We call the union $H_{12}\cup H_{23}\cup H_{31}$ the \emph{spine} of the trisection.  In addition, we observe that $\partial X_i = Y_{k_i} = H_{ij} \cup_\Sigma H_{li}$ is a genus $g$ Heegaard splitting.  Because there is a unique way to cap off $Y_{k_i}$ with $\natural^{k_i}(S^1\times B^3)$~\cite{Laudenbach-Poenaru_A-note_1972,Montesinos_Heegaard_1979},  every trisection is uniquely determined by its spine.

Like Heegaard splittings, trisections can be encoded with diagrams.  A \emph{cut system} for a genus $g$ surface $\Sigma$ is a collection of $g$ pairwise disjoint simple closed curves that cut $\Sigma$ into a $2g$-punctured sphere.  A cut system $\delta$ is said to \emph{define a handlebody} $H_{\delta}$ if each curve in $\delta$ bounds a disk in $H_{\delta}$.  A triple $(\alpha,\beta,\gamma)$ of cut systems is called a \emph{$(g;k_1,k_2,k_3)$--trisection diagram} for $\Tt$ if $\alpha$, $\beta$, and $\gamma$ define the components $H_{\alpha}, H_{\beta}$, and $H_{\gamma}$ of the spine of $\Tt$.  We set the convention that $H_{\alpha} = X_3\cap X_1$, $H_{\beta} = X_1\cap X_2$, and $H_{\gamma} = X_2\cap X_3$.  The careful reader will note that this convention differs slightly from~\cite{Meier-Schirmer-Zupan_Classification_2016}.  With these conventions, $(\alpha,\beta)$, $(\beta,\gamma)$, and $(\gamma,\alpha)$ are Heegaard diagrams for $Y_{k_1}$, $Y_{k_2}$, and $Y_{k_3}$, respectively.  In~\cite{Gay-Kirby_Trisecting_2016}, Gay and Kirby prove that every smooth 4-manifold admits a trisection, and trisection diagrams, modulo handle slides within the three collections of curves, are in one-to-one correspondence with trisections.

\begin{examples}\label{exs:trisections}
	Trisections with genus at most two are well-understood. See Figure~\ref{fig:Diags}.
	\begin{enumerate}
		\item There is a unique genus zero trisection; the $(0,0)$--trisection describing $S^4$.
		\item There are exactly six genus one trisections.  Both $\CP^2$ and $\overline{\CP}^2$ admit $(1;0)$--trisections; $S^1\times S^3$ admits a $(1;1)$--trisection; and $S^4$ admits three unbalanced genus one trisections.
		\item There is a unique irreducible (defined below) genus two trisection~\cite{Meier-Schirmer-Zupan_Classification_2016,Meier-Zupan_Genus-two_2017}, which describes $S^2\times S^2$.
	\end{enumerate}
\end{examples}

Given trisections $\Tt$ and $\Tt'$ for 4-manifolds $X$ and $X'$, we can obtain a trisection for $X\#X'$ by removing a neighborhood of a point in each trisection surface and gluing pairs of components of $\Tt$ and $\Tt'$ along the boundary of this neighborhood.  The resulting trisection is uniquely determined in this manner; we denote it by $\Tt\#\Tt'$.  A trisection $\Tt$ is called \emph{reducible} if $\Tt = \Tt'\#\Tt''$, where neither $\Tt'$ nor $\Tt''$ is the genus zero trisection; otherwise, it is called \emph{irreducible}.  Equivalently, $\Tt$ is reducible precisely when there exists a curve $\delta$ in $\Sigma$ that bounds compressing disks in $H_{\A}$, $H_{\n}$, and $H_{\g}$.  Such a curve $\delta$ represents the intersection of a decomposing 3-sphere with the trisection surface.

\begin{figure}[h!]
	\centering
	\includegraphics[width=.8\textwidth]{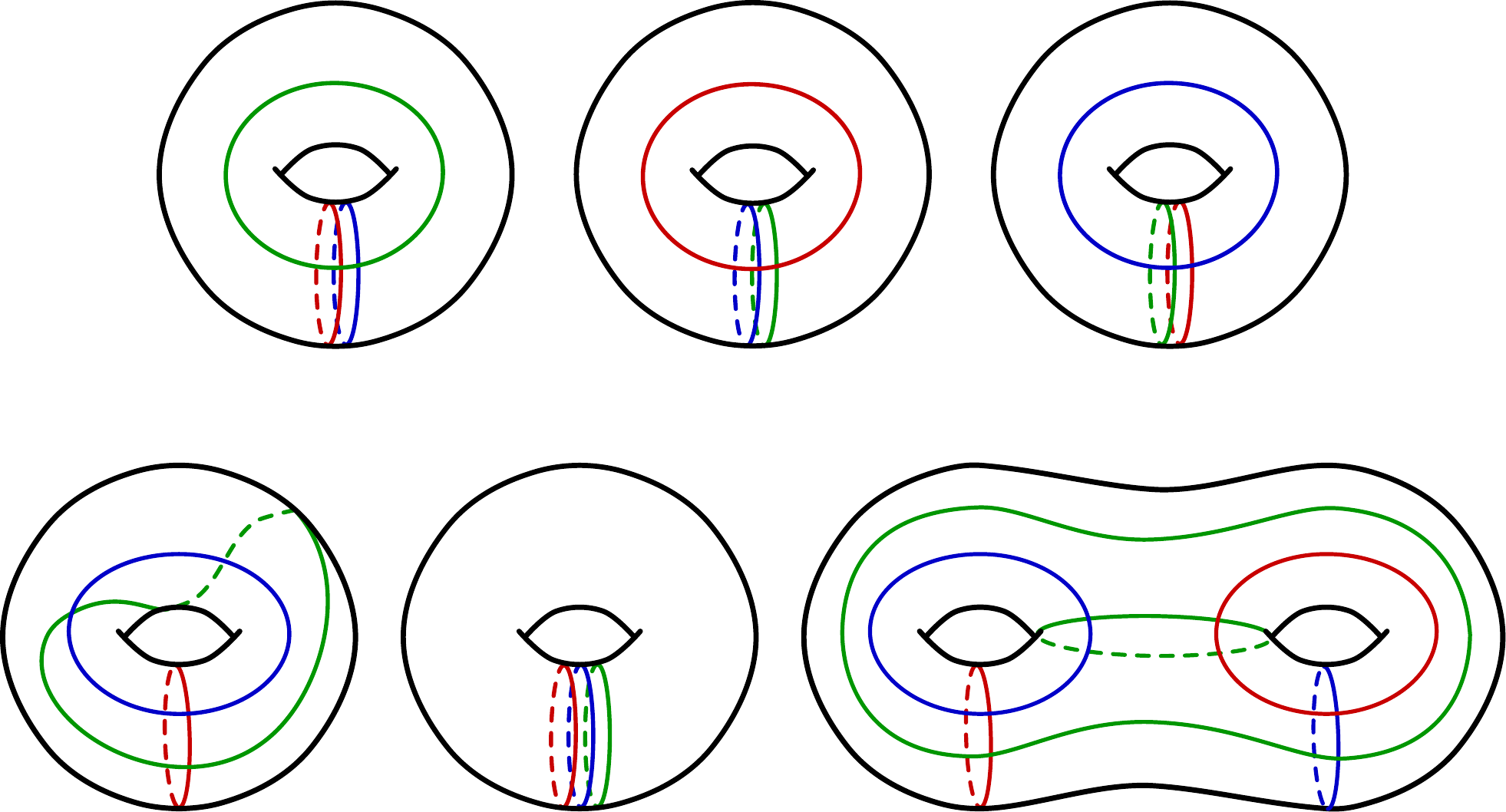}
	\caption{Low-genus trisection diagrams. Top, from left to right: the genus one diagrams $\Ss_1$, $\Ss_2$, and $\Ss_3$ for $S^4$.  Bottom, from left to right: the genus one diagrams for $\CP^2$ and $S^1\times S^3$, and the genus two diagram for $S^2\times S^2$.}
	\label{fig:Diags}
\end{figure}

In dimension three, stabilization of a Heegaard surface may be viewed as taking the connected sum with the genus one splitting of $S^3$, and a similar structure exists for trisections.  Let $\Ss_i$ denote the unique genus one trisection of $S^4$ satisfying $k_i=1$.  Diagrams for these three trisections are shown in Figure~\ref{fig:Diags}.  A trisection $\Tt$ is called \emph{$i$--stabilized} if $\Tt = \Tt'\#\Ss_i$, and is simply called \emph{stabilized} if it is $i$--stabilized for some $i=1,2,3$.  Two trisections $\Tt'$ and $\Tt''$ are called \emph{stably equivalent} if there is a trisection $\Tt$ that is a stabilization of both $\Tt'$ and $\Tt''$.  Gay and Kirby proved that any two trisections of a fixed 4-manifold are stably equivalent~\cite{Gay-Kirby_Trisecting_2016}.

We say that a trisection $\Tt$ is \emph{standard} if $\Tt$ can be expressed as the connected sum of the trisections listed in Examples~\ref{exs:trisections}.  Theorems in~\cite{Meier-Schirmer-Zupan_Classification_2016,Meier-Zupan_Genus-two_2017} classify trisections of genus two.

\begin{theorem}\label{thm:g2standard}
	Every trisection $\Tt$ with genus $g =2$ is standard.
\end{theorem}

%In the special case $X = S^4$, one might wonder if a stronger statement could be true.  A trisection is called \emph{standard} if it is a stabilization of the genus zero trisection of $S^4$, i.e., if it is a connected sum of copies of the $\Ss_i$.  In dimension three, it is a well-known theorem of Waldhausen that every Heegaard splitting of $S^3$ is a stabilization of the genus zero splitting~\cite{Waldhausen_Heegaard-Zerlegungen_1968}, inspiring the following conjecture.

%\begin{conjecture}\cite{Meier-Schirmer-Zupan_Classification_2015}\label{conj:Stand}
%	Every trisection of $S^4$ is standard.
%\end{conjecture}

%As an intermediary between Theorem~\ref{thm:GKstab} and Conjecture~\ref{conj:Stand}, we consider partial stabilizations, in which we take connected sums with some (but not all) of the standard genus one splittings of $S^4$.  Specifically, for any subset $I\subset\{1,2,3\}$, a trisection $\Tt$ of $S^4$ is called \emph{$I$--standard} if $\Tt$ becomes standard after some number of stabilizations using $\Ss_j$ with $j\in I$.  For example, $\Tt$ is 1--standard if only 1--stabilizations are required to make $\Tt$ standard and $\{2,3\}$--standard if 1--stabilizations are \emph{not} required.  By Theorem~\ref{thm:GKstab}, every trisection of $S^4$ is $\{1,2,3\}$--standard.

Below, we see how this theorem implies Theorem~\ref{main1}, and for this purpose, we turn our attention to surgery on links.  

%%%%%%%%%%%%%%%%%%%%%%%%%%%%%%%%%%%%%%%%%%%%%%%%%%%%%%%%%%%%%%%%%%%%%%%%%%%%%
\subsection{Admissible links and surfaces}\label{subsec:admissible}\ 
%%%%%%%%%%%%%%%%%%%%%%%%%%%%%%%%%%%%%%%%%%%%%%%%%%%%%%%%%%%%%%%%%%%%%%%%%%%%%

Recall that $Y_k$ denotes $\#^k(S^1 \times S^2)$, and let $L$ be a framed $n$--component link in $Y_k$ such that Dehn surgery on $L$ yields $Y_{k'}$.  We call such a link \emph{admissible}.  If $L$ is an admissible link, $L$ describes a closed 4-manifold $X_L$ with a handle decomposition with $k$ 1--handles, $n$ 2--handles, and $k'$ 3--handles.  An \emph{admissible} Heegaard surface $\Sigma$ for $L$ is a Heegaard surface cutting $Y_k$ into two handlebodies $H$ and $H'$, such that a core of $H$ contains $L$.  As such, $C = H \setminus \nu(L)$ is a compression body and $\Sigma$ may be viewed as a Heegaard surface for the link exterior $E(L) = Y_k \setminus \nu(L)$.  Let $H_L$ be the handlebody that results from Dehn filling $C$ (or performing Dehn surgery on $L$ in $H$) along the framing of the link $L$.  An \emph{admissible pair} consists of an admissible link together with an admissible Heegaard surface.

For completeness, we will also allow the empty link, $L = \emp$.  An admissible surface $\Sigma$ for the empty link is a (standard) genus $g$ Heegaard surface for $Y_k$.  A genus $g$ Heegaard diagram $(\A,\n)$ for $Y_k$ is called \emph{standard} if $\A \cap \n$ contains $k$ curves, and the remaining $g-k$ curves occur in pairs that intersect once and are disjoint from other pairs.  A trisection diagram is called \emph{standard} if each pair is a standard Heegaard diagram.

\begin{lemma}\label{lem:constr}
	Let $L$ be an admissible $n$-component link in $Y_k$.  Every admissible pair $(L,\Sigma)$ gives rise to a trisection $\Tt(L,\Sigma)$ with spine $H' \cup H \cup H_L$.  If $g(\Sigma) = g$, then $\Tt(L,\Sigma)$ is a $(g;k,g-n,k')$-trisection.  Moreover, there is a trisection diagram $(\A,\n,\g)$ for $\Tt(L,\Sigma)$ such that
\begin{enumerate}
\item $H_{\A} = H'$, $H_{\n} = H$, and $H_{\g} = H_L$;
\item $L$ is a sublink of $\g$, where $\g$ is viewed as a link framed by $\Sigma$ in $Y_k = H_{\A} \cup H_{\n}$; and
\item $(\n,\g)$ is a standard diagram for $Y_{g-n}$, where $\n \cap \g = \g \setminus L$.
\end{enumerate}
\end{lemma}

\begin{proof}
This is proved (in slightly different formats) for $L \neq \emp$ in both~\cite{Gay-Kirby_Trisecting_2016} and~\cite{Meier-Schirmer-Zupan_Classification_2016}.  If $L = \emp$, then it follows easily that $X_L$ has a handle decomposition without 2-handles, $H = H_L$, and $H' \cup H \cup H_L$ is the spine for the $(g;k,g,k)$-trisection $\Tt(L,\Sigma)$ of $X_L$.  In this case, there is a diagram such that $\n = \g$, the standard genus $g$ diagram of $Y_g$.
\end{proof}

This machinery is enough to prove Theorem~\ref{main1}, classifying cosmetic surgeries on tunnel number one links in $S^3$.  Note that the conventions $H_{\A} = H'$, $H_{\n} = H$, and $H_{\g} = H_L$ agree with our earlier conventions identifying the union of the 0--handle and the 1--handles with $X_1$, the trace of the Dehn surgery on $H_{\n}$ along $L$ with $X_2$, and the union of the 3--handles and the 4--handle with $X_3$.

\begin{proof}[Proof of Theorem~\ref{main1}]
Suppose $L \subset S^3$ is a tunnel number one link with an integral Dehn surgery to $S^3$.  Then there exists an admissible surface $\Sigma \subset S^3$ and a genus two trisection $\Tt(L,\Sigma)$ with a diagram $(\A,\n,\g)$, where $H_L = H_{\g}$.  By Lemma~\ref{lem:constr}, $\Tt(L,\Sigma)$ is a $(2,0)$-trisection, the two curves in $\g$ are isotopic to the link $L$ in $S^3 = H_{\A} \cup H_{\n}$, and the surface framing of $\g$ in $\Sigma$ is the framing of $L$.  By Theorem~\ref{thm:g2standard}, the trisection $\Tt(L,\Sigma)$ of $X_L$ is standard, and $(\A,\n,\g)$ is handleslide equivalent to a standard diagram $(\A',\n',\g')$.  Since $\Tt(L,\Sigma)$ is $(2,0)$-trisection, $X_L$ is diffeomorphic to either $S^2 \X S^2$ or $\pm \CP \# \pm \CP$.  In the first case, $\g$ is handleslide equivalent to $\g'$, which is a zero-framed Hopf link in $S^3 = H_{\A} \cup H_{\n}$.  In the second case, $\g$ is handleslide equivalent to $\g'$, a 2-component unlink with framings $\pm 1$ and $\pm 1$, completing the proof.
\end{proof}

We now turn our attention to R-links.  Note that if $L$ is an R-link, then the smooth 4-manifold $X_L$ has a handle decomposition with no 1-handles, $n$ 2-handles, and $n$ 3-handles; thus $X_L$ is a simply connected 4-manifold with $\chi(X_L) = 2$, so that $X_L$ is a homotopy $S^4$.  We describe an immediate connection between Stable Property R and trisections in the next lemma.

\begin{lemma}\label{lem:stdslide}
Suppose $L$ is an R-link with admissible surface $\Sigma$ and $\Tt(L,\Sigma)$ is a standard trisection of $S^4$.  Then $L$ has Stable Property R.
\end{lemma}

\begin{proof}
By Lemma~\ref{lem:constr}, the trisection $\Tt(L,\Sigma)$ has a diagram $(\A,\n,\g)$ such that $(\n,\g)$ is the standard Heegaard diagram for $Y_{g-n}$.  Viewing $\g$ as a $g$-component link in $S^3 = H_{\A} \cup H_{\n}$, we have that $(g-n)$ curves in $\g$ bounds disks in $H_{\n}$, while the remaining $n$ curves are isotopic to $L$ (and are disjoint from the $(g-n)$ disks).  Thus, as a link in $S^3$, we have $\g = L \sqcup U_{g-n}$.

In addition, the trisection $\Tt(L,\Sigma)$ is a standard $(g;0,g-n,n)$-trisection of $S^4$ by hypothesis.  As such, it must be a connected sum of $g-n$ copies of $\Ss_2$ and $n$ copies of $\Ss_3$, and it has a standard diagram, $(\A',\n',\g')$, where $g-n$ curves in $\g'$ are also curves in $\n'$, and the remaining $n$ curves are also curves in $\A'$.  Thus, in $S^3 = H_{\A'} \cup H_{\n'}$, the curves $\g'$ comprise a $g$-component unlink, with surface framing equal to the zero framing on each component.  Since $(\A,\n,\g)$ and $(\A',\n',\g')$ are trisection diagrams for the same trisection, we have that $\g$ is handleslide equivalent to $\g'$ via slides contained in $\Sigma$.  Thus, $\g$ and $\g'$ are handleslide equivalent links in $S^3$.  We conclude that $L$ has Stable Property R, as desired.
\end{proof}

Theorem~\ref{main2} can be quickly proved using this lemma and the following result from~\cite{Meier-Schirmer-Zupan_Classification_2016} as its main input.

\begin{theorem}\label{thm:msz}
Every $(g;0,1,g-1)$-trisection is a standard trisection of $S^4$.
\end{theorem}

%Note that if $X$ has a $(g;0,1,g-1)$-trisection $\Tt$, then $X$ has a handle decomposition with no 1--handles, $(g-1)$ 2--handles, and $(g-1)$ 3--handles.  As such, $X$ is a homotopy 4-sphere and the standard trisection $\Tt$ is a stabilization of the genus zero trisection of $S^4$.

\begin{proof}[Proof of Theorem~\ref{main2}]
Suppose $L \subset S^3$ is an $n$-component link with tunnel $n$ with a Dehn surgery to $\#^n(S^1 \X S^2)$.  Then by Lemma~\ref{lem:constr} there exists an admissible surface $\Sigma \subset S^3$ and an $(n+1;0,1,n)$-trisection $\Tt(L,\Sigma)$.  By Theorem~\ref{thm:msz}, the trisection $\Tt(L,\Sigma)$ is standard, and by Lemma~\ref{lem:stdslide}, $L$ has Stable Property R.  In fact, the proof of Lemma~\ref{lem:stdslide} reveals that $L \sqcup U_1$ has Property R, as desired.
\end{proof}

%%%%%%%%%%%%%%%%%%%%%%%%%%%%%%%%%%%%%%%%%%%%%%%%%%%%%
%%%%%%%%%%%%%%%%%%%%%%%%%%%%%%%%%%%%%%%%%%%%%%%%%%%%%
\section{R-links and stabilizations}\label{sec:R-stab}
%%%%%%%%%%%%%%%%%%%%%%%%%%%%%%%%%%%%%%%%%%%%%%%%%%%%%
%%%%%%%%%%%%%%%%%%%%%%%%%%%%%%%%%%%%%%%%%%%%%%%%%%%%%

In order to prove the third main theorem, we will further develop the connection between R-links, their induced trisections, and the various stabilization operations.

\begin{lemma}\label{lem:emp}
If $L = \emp$ in $Y_k$ and $g(\Sigma) = g$, then $X_L = \#^k(S^1 \times S^3)$, and $\Tt(\emp,\Sigma)$ is the connected sum of $k$ copies of the standard $(1;1)$-trisection of $S^1 \times S^3$ and $g-k$ copies of $\Ss_2$.
\end{lemma}
\begin{proof}
By Waldhausen's Theorem~\cite{Waldhausen_Heegaard-Zerlegungen_1968}, $Y_k$ has a standard Heegaard diagram, $(\A,\n)$, and by Lemma~\ref{lem:constr}, $(\A,\n,\n)$ is a trisection diagram for $\Tt(\emp,\Sigma)$.  The $k$ curves in $\A \cap \n$ give rise to $k$ summands of the standard genus one splitting of $S^1 \times S^3$, and the remaining $g-k$ pairs give rise to $g-k$ copies of $\Ss_2$.
\end{proof}

%\begin{lemma}\label{lem:stable}
%	Suppose that $(L,\Sigma)$ is an admissible pair and $\Sigma'$ is a stabilization of $\Sigma$.  Then $(L,\Sigma')$ is an admissible pair and $\Tt(L,\Sigma')$ is a 1--stabilization of $\Tt(L,\Sigma)$.
%\end{lemma}
%
%\begin{proof}
%	Stabilizing $\Sigma$ does not interfere with the link $L$ as a subset of a core of $H$, so we need only check that this process corresponds to a 1--stabilization of $\Tt(L,\Sigma)$.  Note that stabilizing $\Sigma$ yields new compressing disks $D_{\A}$ and $D_{\n}$ for $H_{\A}$ and $H_{\n}$ that meet in a single point.  However, $D_{\A}$ is also a compressing disk for $C$, and as such, $\pd D_{\A}$ bounds a disk in $H_{\g}$ as well.  It follows that $\Tt(L,\Sigma)$ is related to $\Tt(L,\Sigma')$ by a 1--stabilization.
%\end{proof}

In order to understand operations on an admissible link $L$ and Heegaard surface $\Sigma$ which will correspond to various stabilizations of $\Tt(L,\Sigma)$, we introduce several additional definitions.  Let $(L_1,\Sigma_1) \subset Y_{k_1}$ and $(L_2,\Sigma_2) \subset Y_{k_2}$ be any admissible pairs, and define the operation $\ast$ by
\[ (L_1,\Sigma_1) \ast (L_2,\Sigma_2) = (L_1 \sqcup L_2, \Sigma_1 \# \Sigma_2),\]
where the connected sum is taken so that $L_1 \sqcup L_2$ is not separated by the surface $\Sigma_1 \# \Sigma_2$.  Note that $(L_1,\Sigma_1) \ast (L_2,\Sigma_2) \subset Y_{k_1+k_2}$.

\begin{lemma}\label{lem:sum}
If $(L_1,\Sigma_1)$ and $(L_2,\Sigma_2)$ are admissible pairs, then $(L,\Sigma) = (L_1,\Sigma_1) \ast (L_2,\Sigma_2)$ is an admissible pair, and $\Tt(L,\Sigma) = \Tt(L_1,\Sigma_1) \# \Tt(L_2,\Sigma_2)$.
\end{lemma}
\begin{proof}
It is clear that the framed link $L_1 \sqcup L_2 \subset Y_{k_1+k_2}$ has the appropriate surgery.  Suppose $\Sigma_i$ bounds a handlebody $H_i$ with core $C_i$ containing $L_i$.  Then there is a core $C$ for $H_1 \natural H_2$ such that $L_1 \sqcup L_2 \subset C_1 \sqcup C_2 \subset C$, and thus $\Sigma_1 \# \Sigma_2$ is admissible as well.  For the second claim, note that the curve $\delta$ arising from the connected sum $\Sigma = \Sigma_1 \# \Sigma_2$ is a reducing curve for $\Tt(L,\Sigma)$, splitting it into the trisections $\Tt(L_1,\Sigma_1)$ and $\Tt(L_2,\Sigma_2)$.
\end{proof}

Let $U$ be a 0-framed unknot in $S^3$, and let $\Sigma_U$ be the genus one splitting of $S^3$ such that one of the solid tori bounded by $\Sigma_U$ contains $U$ as a core.  In addition, let $W$ denote the knot $S^1 \times \{\text{pt}\} \subset S^1 \times S^2$ with framing given by the fibering, and let $\Sigma_W$ be the genus one splitting of $S^1 \times S^2$ such that one of the solid tori bounded by $\Sigma_W$ contains $W$ as a core.  Note that both $(U,\Sigma_U)$ and $(W,\Sigma_W)$ are admissible pairs.  Finally, let $\Sigma_{\emp}$ be the genus one Heegaard surface for $S^3$, to be paired with the empty link.

\begin{lemma}\label{lem:g1}
The links $(W,\Sigma_W)$, $(\emp,\Sigma_{\emp})$, and $(U,\Sigma_U)$ yield the following trisections:
{\begin{enumerate}
\item $\Tt(W,\Sigma_W) = \Ss_1$.
\item $\Tt(\emp,\Sigma_{\emp}) = \Ss_2$.
\item $\Tt(U,\Sigma_U) = \Ss_3$.
\end{enumerate}}
\end{lemma}
\begin{proof}
First, note that each trisection in question has genus one.  Since framed surgery on $W \subset Y_1$ yields $S^3$, by Lemma~\ref{lem:constr}, $\Tt(W,\Sigma_W)$ is a $(1;1,0,0)$--trisection and must be $\Ss_1$. Similarly, $\Tt(\emp,\Sigma_{\emp})$ is a $(1;0,1,0)$--trisection and must be $\Ss_2$.  Finally, framed surgery on $U$ yields $Y_1$, so $\Tt(U,\Sigma_U)$ is a $(1;0,0,1)$--trisection and must be $\Ss_3$.
\end{proof}

By combining Lemmas~\ref{lem:emp},~\ref{lem:sum}, and~\ref{lem:g1}, we obtain
\begin{corollary}\label{cor:stab}
Suppose $(L,\Sigma)$ is an admissible link, with $\Tt = \Tt(L,\Sigma)$.
{\begin{enumerate}
\item $\Tt((L,\Sigma) \ast (W,\Sigma_W))$ is the 1--stabilization of $\Tt$.
\item $\Tt((L,\Sigma) \ast (\emp,\Sigma_\emp))$ is the 2--stabilization of $\Tt$.
\item $\Tt((L,\Sigma) \ast (U,\Sigma_U))$ is the 3--stabilization of $\Tt$.
\end{enumerate}}
In addition, if $\Sigma_+$ is the stabilization of $\Sigma$ (as a Heegaard surface for $Y_k$), then $(L,\Sigma_+) = (L,\Sigma) \ast (\emp,\Sigma_{\emp})$.
\end{corollary}

We say that two trisections $\Tt_1$ and $\Tt_2$ of a 4-manifold $X$ are \emph{2--equivalent} if there is a trisection $\Tt$ that is the result of 2--stabilizations performed on both $\Tt_1$ and $\Tt_2$.

\begin{lemma}\label{lem:equiv}
	If $\Sigma_1$ and $\Sigma_2$ are two distinct admissible surfaces for an admissible link $L$, then the trisections $\Tt(L,\Sigma_1)$ and $\Tt(L,\Sigma_2)$ are 2--equivalent.
\end{lemma}

\begin{proof}
	Since both $\Sigma_1$ and $\Sigma_2$ are Heegaard surfaces for $E(L)$, they have a common stabilization $\Sigma$ by the Reidemeister-Singer Theorem~\cite{Reidemeister_Zur-dreidimensionalen_1933,Singer_Three-dimensional_1933}.  By Lemma~\ref{lem:sum}, the surface $\Sigma$ is admissible, and by Corollary~\ref{cor:stab}, $\Tt(L,\Sigma)$ can be obtained by 2--stabilizations of $\Tt(L,\Sigma_i)$.
\end{proof}

Since 2--equivalence is an equivalence relation, Lemma \ref{lem:equiv} implies that every admissible surface $\Sigma$ for an admissible link $L$ belongs to the same 2--equivalence class.  Hence, $L$ has a well-defined \emph{2--equivalence class}; namely, the 2--equivalence class of $\Tt(L,\Sigma)$.  If two admissible links $L_1$ and $L_2$ give rise to 2--equivalent trisections, we say that $L_1$ and $L_2$ are \emph{2--equivalent}.

Suppose that $L$ is an $n$-component admissible link with admissible surface $\Sigma$, so that $\Sigma$ cuts $Y_k$ into $H \cup H'$, and $L$ is isotopic into a core $C \subset H$.  As such, there is a collection of $n$ compressing disks $\{D\}$ with the property that each disk meets a unique component of $L$ once and misses the other components.  We call $\{D\}$ a set of \emph{dualizing disks}.  Note that if $(\A,\n,\g)$ is the trisection diagram for $\Tt(L,\Sigma)$ guaranteed by Lemma~\ref{lem:constr}, then the $n$ disks bounded by the $n$ curves in $\n$ that are not in $\g$ are a set of dualizing disks for $L$.

\begin{lemma}\label{lem:slide}
	If admissible links $L_1$ and $L_2$ are related by a handleslide, then $L_1$ and $L_2$ are 2--equivalent.
\end{lemma}
\begin{proof}
	If $L_i$ is an $n$--component link, then $L_1$ and $L_2$ have $n-1$ components in common and differ by a single component, $L_1' \subset L_1$ and $L_2' \subset L_2$, where a slide of $L_1'$ over another component $L'$ of $L_1$ along a framed arc $a$ yields $L_2'$.  Consider $\Gamma = L_1 \cup a$, an embedded graph with $n-1$ components, and let $\Sigma$ be a Heegaard surface cutting $S^3$ into $H \cup H'$, where $\Gamma$ is contained in a core of $H$.  Then $L_1$ is also contained in a core of $H$, and $\Sigma$ is admissible (with respect to $L_1$).  Let $\{D_1\}$ be a set of dualizing disks for $L_1$.  A priori, the arc $a$ might meet some of the disks in $\{D_1\}$; however, if this is the case, we can perform a sequence of stabilizations on $\Sigma$, after which $a$ avoids all of the disks $\{D_1\}$.  Thus, we suppose without loss of generality that $a \cap \{D_1\} = \emptyset$.
	
There is an isotopy taking $\Gamma$ into $\Sigma$, preserving the intersections of $L_i$ with the dualizing disks $\{D_1\}$, so that the framing of $\Gamma$ agrees with its surface framing in $\Sigma$.  As such, we can perform the handleslide of $L_1'$ over $L'$ along $a$ within the surface $\Sigma$, so that the resulting link $L_2$ is also contained in $\Sigma$, with framing given by the surface framing.  Let $D_1' \in \{D_1\}$ be the disk that meets $L_1'$ once, and let $D' \in \{D_1\}$ be the disk that meets $L'$ once.  There is an arc $a'$, isotopic in $\Sigma$ to an arc in $\Gamma$, that connects $D_1'$ to $D'$. See Figure~\ref{fig:slide}. Let $D_2'$ be the compressing disk obtained by banding $D_1'$ to $D'$ along $a'$.  Then $\{D_2\} = (\{D_1\} \setminus D')\cup D_2'$ is a set of dualizing disks for $L_2$. Thus, by pushing $L_2$ back into $H$, we see that $\Sigma$ is an admissible surface for $L_2$.

\begin{figure}[h!]
	\centering
	\includegraphics[width=.5\textwidth]{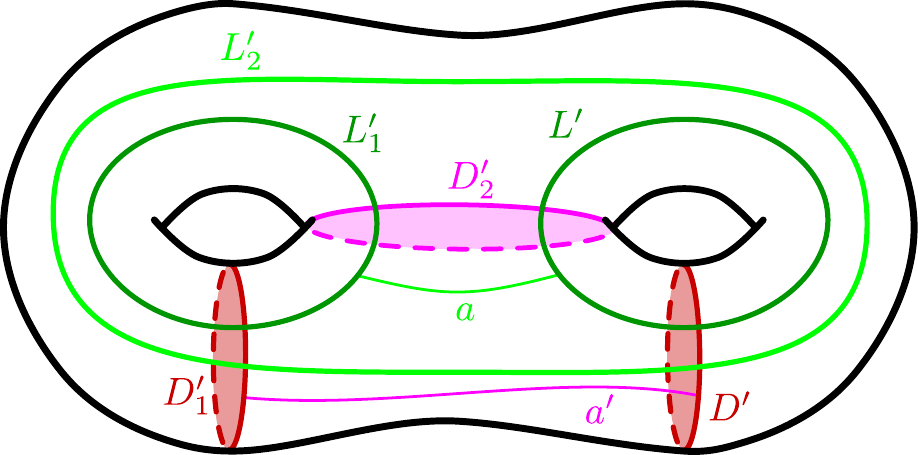}
	\caption{A schematic diagram showing how one can adjust the disk system $\{D_1\}$ dual to $L_1$ to get a disk system $\{D_2\}$ dual to $L_2$ when $L_2$ is obtained from $L_1$ via a surface framed handleslide among components $L_1'$ and $L'$ of $L_1'$.}
	\label{fig:slide}
\end{figure}

Following Lemma \ref{lem:constr}, let $H_i \cup H_i' \cup H_{L_i}$ be a spine for $\Tt(L_i,\Sigma)$.  By construction, $H_1 = H_2$ and $H_1' = H_2'$.  Finally, since $H_i$ is Dehn surgery on $L_i$ in $H_i$ , and $L_1$ and $L_2$ are related by a single handleslide, we have $H_{L_1} = H_{L_2}$.  It follows that $\Tt(L_1,\Sigma) = \Tt(L_2,\Sigma)$, and we conclude that $L_1$ and $L_2$ are 2--equivalent.
\end{proof}

For the rest of the section, we will restrict our attention to admissible links in $S^3$.  Let $U_n$ denote the zero-framed, $n$--component unlink, so $X_{U_n} = S^4$.  Recall that a \emph{standard trisection} of $S^4$ is the connected sum of copies of $\Ss_1$, $\Ss_2$, and $\Ss_3$.

\begin{lemma}\label{lem:unlink}
	Let $\Sigma$ be any admissible surface for $U_n$, then $\Tt(U_n,\Sigma)$ is standard.
\end{lemma}
\begin{proof}
We induct on $(n,g)$ with the dictionary ordering.  If $n=1$, then $E(U_1)$ is a solid torus.  If $g=1$, then $\Sigma = \Sigma_U$, so that $\Tt(U_1,\Sigma_U) = \Ss_3$ by Lemma~\ref{lem:g1}.  If $n=1$ and $g > 1$, then $\Sigma$ is stabilized~\cite{lei} (see also~\cite{scharlemann-thompson}), which means that $\Tt(U_1,\Sigma)$ is 2--stabilized by Corollary~\ref{cor:stab}, and as such, $\Tt(U_1,\Sigma)$ is standard by induction.

In general, note that the Heegaard genus of an $n$--component unlink is $n$; thus $g \geq n$ for all possible pairs $(n,g)$.  For $n>1$, we have that $E(U)$ is reducible, and so Haken's Lemma~\cite{Haken_Some_1968} implies that $\Sigma$ is reducible, splitting into the connected sum of genus $g_1$ and $g_2$ surfaces $\Sigma_1$ and $\Sigma_2$, where $\Sigma_i$ is a Heegaard surface for $E(U_{n_i})$.  Then $\Tt(U_n,\Sigma) = \Tt(U_{n_1},\Sigma_1) \# \Tt(U_{n_2},\Sigma_2)$, where $(n_i,g_i) < (n,g)$.  Since both summands are standard trisections by induction, it follows that $\Tt(U_n,\Sigma)$ is also standard, completing the proof.
\end{proof}

%If $L$ is a framed $n$--component link in $S^3$ that admits an integral Dehn surgery yielding $Y_n$, then we call $L$ a \emph{R-link}. The only possible framing of such a surgery is the canonical zero-framing on $L$ that is induced, for each component, by a Seifert surface for that component in $S^3$.  Note that $U_n$ is an R-link, as is any link that is handleslide equivalent to $U_n$.  We say that such an $R$-link $L$ has \emph{Property R}.  If $L\sqcup U_r$ has Property R for some integer $r$, then we say that $L$ has \emph{Stable Property R}.

A trisection $\Tt$ is said to be \emph{2--standard} if it becomes standard after some number of 2-stabilizations.  Similarly, $\Tt$ is $\{2,3\}$--standard if it becomes standard after some number of 2- and 3-stabilizations.

\begin{proof}[Proof of Theorem~\ref{thm:equiv}]
Suppose $L$ has Property R. By Lemma \ref{lem:slide}, $L$ and $U_n$ are 2--equivalent links.  Thus, $\Tt(L,\Sigma)$ is 2--equivalent to some trisection coming from $U_n$, but all trisections induced by $U_n$ are standard by Lemma \ref{lem:unlink}, and thus $\Tt(L,\Sigma)$ becomes standard after a finite sequence of 2--stabilizations.
	
If $L$ has Stable Property R, then $L \sqcup U_n$ has Property R for some $n$, and thus $\Tt((L,\Sigma) \ast (U,\Sigma_U) \ast \dots \ast (U,\Sigma_U))$ is 2--standard by the above arguments.  By Lemma~\ref{lem:g1} and Corollary~\ref{cor:stab},
\[\Tt((L,\Sigma) \ast (U,\Sigma_U) \ast \dots \ast (U,\Sigma_U)) = \Tt(L,\Sigma) \# \Ss_3 \# \dots \# \Ss_3;\]
hence $\Tt(L,\Sigma)$ is $\{2,3\}$--standard.

Finally, if the trisection $\Tt(L,\Sigma)$ is $\{2,3\}$--standard, then there exist integers $s$ and $t$ such that the connected sum of $\Tt(L,\Sigma)$ with $s$ copies of $\Ss_2$ and $t$ copies of $\Ss_3$ is standard.  Let $(L_*,\Sigma_*)$ be the admissible pair given by
\[ (L_*,\Sigma_*) = (L,\Sigma) \ast \underbrace{(\emp,\Sigma_{\emp}) \ast \dots \ast (\emp,\Sigma_{\emp})}_s  \ast \underbrace{(U,\Sigma_U) \ast \dots \ast (U,\Sigma_U)}_t.\]
By assumption, $\Tt(L_*,\Sigma_*)$ is standard, so by Lemma~\ref{lem:stdslide}, the link $L_*$ has Stable Property R.  But by definition of $\ast$, we have $L_* = L \sqcup U_t$, and thus $L$ also has Stable Property R, completing the proof.
\end{proof}

%We now state the family of conjectures which motivate the preceding theorem and much of the remainder of this article.

%\begin{GPRC}
%	Let $L$ be a link in $S^3$.
%	\begin{description}
%		\item[Strong] If $L$ is an R-link, then $L$ has Property R.
%		\item[\hspace{.035in}Stable] If $L$ is an R-link, then $L$ has stable Property R.
%		\item[\hspace{.0875in}Weak] If $L$ is an R-link, then $L$ has weak Property R.
%	\end{description}
%\end{GPRC}

%Clearly, (1) implies (2), which implies (3).  These conjectures have a number of intriguing connections to other open questions.

%\begin{proposition}[\cite{Scharlemann_Generalized_2008}, Prop 2.4]
%	The Weak Generalized Property R Conjecture is equivalent to the conjecture that any homotopy four-sphere (resp., 4-ball) that can be built without 3--handles is diffeomorphic to $S^4$ (resp., $B^4$).
%\end{proposition}

%In particular, the Weak GPRC is equivalent to the Smooth Poincar\'e Conjecture for homotopy four-spheres built without 3--handles~\cite{Gompf-Scharlemann-Thompson_Fibered_2010}.  Scharlemann also outlines an approach to the Sch\"onflies Conjecture that is closely related to the trisection viewpoint that we adopt.  In particular, see Propositions~5.1~and~5.3 of~\cite{Scharlemann_Generalized_2008}, as well as Subsection~\ref{subsec:critical} below.

%%%%%%%%%%%%%%%%%%%%%%%%%%%%%%%%%%%%%%%%%%%%%%%%%%%%%%%%%%%%%%%%%%%%%%%%%%%%
%%%%%%%%%%%%%%%%%%%%%%%%%%%%%%%%%%%%%%%%%%%%%%%%%%%%%%%%%%%%%%%%%%%%%%%%%%%%
\section{Trisecting the Gompf-Scharlemann-Thompson Examples}\label{sec:GST}
%%%%%%%%%%%%%%%%%%%%%%%%%%%%%%%%%%%%%%%%%%%%%%%%%%%%%%%%%%%%%%%%%%%%%%%%%%%%
%%%%%%%%%%%%%%%%%%%%%%%%%%%%%%%%%%%%%%%%%%%%%%%%%%%%%%%%%%%%%%%%%%%%%%%%%%%%
Although this view has changed in the past and may change in the future, it is the current view of the authors that the GPRC is likely false.  In light of this opinion, we will outline the first steps one might take to employ Theorem~\ref{thm:equiv} to disprove the GRPC or the Stable GPRC.  Let $L$ be an R-link with admissible surface $\Sigma$.  By Theorem~\ref{thm:equiv}, if $\Tt(L,\Sigma)$ is \textbf{not} $\{2\}$--standard, then $L$ fails to have Property R.  Thus, we in this section we will show how to take the most promising potential counterexamples to the GPRC and construct admissible surfaces and their corresponding trisections.

The possible counterexamples mentioned in the previous paragraph were produced by Gompf-Scharlemann-Thompson~\cite{Gompf-Scharlemann-Thompson_Fibered_2010}, building on work of Akbulut-Kirby~\cite{Akbulut-Kirby_A-potential_1985}.   We will call this family the \emph{GST links}.  In order to describe the construction of the GST links, we need several preliminary details.  Let $Q$ denote the square knot, the connected sum of the right-handed and left-handed trefoil knots, and let $F$ denote the genus two fiber surface for the square knot.  In~\cite{Scharlemann_Proposed_2012}, Scharlemann depicted an elegant way to think about the monodromy corresponding to the fibration of $E(Q)$ by $F$:  We may draw $F$ as a topological annulus $A$, and such that
\begin{itemize}
\item A disk $D$ has been removed from $A$,
\item each component of $\pd A$ is split into six edges and six vertices, and
\item opposite inside edges and opposite outside edges of $\pd A$ are identified to form $F$.
\end{itemize}

With respect to $A$, the monodromy $\varphi$ is a 1/6th clockwise rotation of $A$, followed by an isotopy of $D$ returning it to its original position.  Let $Y_Q$ be the closed 3-manifold obtained by 0-surgery on $Q$, so that $Y_Q$ is a fibered 3-manifold with fiber surface $\wh F$ and monodromy $\wh \varphi$, called the \emph{closed monodromy} of $Q$. Note that the monodromy $\wh F$ is an honest 1/6th rotation of the annulus in Figure~\ref{fig:Hexulus}, since, in this case, the puncture has been filled in by the Dehn surgery.  Details can be found in~\cite{Gompf-Scharlemann-Thompson_Fibered_2010} and~\cite{Scharlemann_Proposed_2012}, where the following lemma is first proved.

\begin{lemma}
For every rational number $p/q$ with $q$ odd, there is a family $\{V_{p/q},V'_{p/q},V''_{p,q}\}$ of curves contained in $\wh F$ that are permuted by $\wh \varphi$.
\end{lemma}

\begin{proof}
We may subdivide $A$ into six rectangular regions as shown in Figure~\ref{fig:Hexulus}.  It is proved in~\cite{Scharlemann_Proposed_2012} that $\wh F$ is a 3-fold branched cover of a 2-sphere $S$ with four branch points.  By naturally identifying $S$ with 4-punctured sphere constructed by gluing two unit squares along their edges, there is a unique isotopy class of curve $c_{p/q}$ with slope $p/q$ in $S$.  Let $\rho:F \rightarrow S$ denote the covering map.  Scharlemann proves that $\rho^{-1}(c_{p,q}) = \{V_{p/q},V'_{p/q},V''_{p,q}\}$, and these curves are permuted by $\wh \varphi$.
\end{proof}

Figure~\ref{fig:Hexulus} shows the three lifts, $V_{3/7}$, $V_{3/7}'$, and $V_{3/7}''$, of the rational curve $3/7$ to the fiber $F$ of the square knot.  Note that $\wh \varphi\, ^6$ is the identity map, and $\wh \varphi\,^3$ maps $V_{p/q}$ to itself but with reversed orientation.

\begin{figure}[h!]
	\centering
	\includegraphics[width=.4\textwidth]{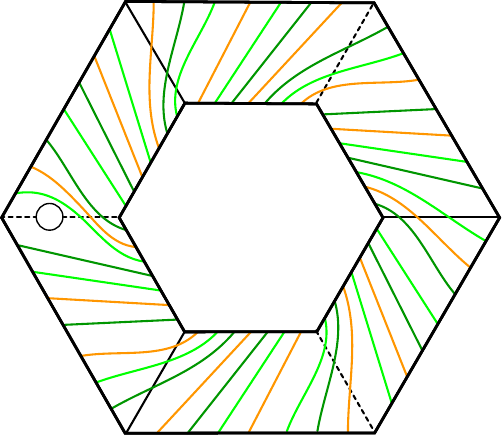}
	\caption{The curves $V_{3/7}$, $V_{3/7}'$, and $V_{3/7}''$ on the genus two fiber $F$ for the square knot.}
	\label{fig:Hexulus}
\end{figure}

Finally, we can define the GST links.  The next lemma is also from~\cite{Scharlemann_Proposed_2012}.

\begin{lemma}
The GST link $L_n$ is handleslide equivalent to $Q \cup V_{n/2n+1}$.  The R-link $L_n$ has Property R for $n = 0,1,2$ and is not known to have Property R for $n \geq 3$.
\end{lemma}

For ease of notation, let $V_n = V_{n/2n+1}$ and $V_n' = V'_{n/2n+1}$, so that $L_n = Q \cup V_n$.  Two links $L$ and $L'$ are said to be \emph{stably handleslide equivalent} or just \emph{stably equivalent} if there are integers $n$ and $n'$ so that $L \sqcup U_n$ is handleslide equivalent to $L' \sqcup U_{n'}$.  While we can find admissible surfaces for $L_n$, there is a simpler construction for a family of links $L_n'$ stably equivalent to $L_n$ for each $n$, and we note a link $L$ has Stable Property R if and only if every link stably equivalent to $L$ has Stable Property R.

\begin{lemma}
The link $L_n = Q \cup V_n$ is stably equivalent to $L'_n = V_n \cup V_n'$.
\end{lemma}
\begin{proof}
We will show that both links are stably equivalent to $Q \cup V_n \cup V_n'$.  Since $\wh \varphi(V_n) = V_n'$, we have that $V'_n$ is isotopic to $V_n$ in $Y_Q$.  Carrying this isotopy into $S^3$, we see that after some number of handleslides of $V_n'$ over $Q$, the resulting curve $C'$ is isotopic to $V_n$.  Now $C'$ can be slide over $V_n$ to produce a split unknot $U_1$, and $Q \cup V_n \cup V_n'$ is handleslide equivalent to $L_n \sqcup U_1$.  On the other hand, $V_n$ and $V_n'$ are homologically independent in the genus two surface $F$.  Thus, there is a sequence of slides of $Q$ over $V_n$ and $V_n'$ converting it to a split unknot, so $Q \cup V_n \cup V_n'$ is handleslide equivalent to $L_n' \sqcup U_1$ as well.
\end{proof}

Next, we will define an admissible surface for $L_n'$.  Consider a collar neighborhood $F \X I$ of $F$, and let $N \subset S^3$ denote the embedded 3-manifold obtained by crushing $\pd F \X I$ to a single curve.  Letting $\Sigma = \pd N$, we see that $\Sigma$ is two copies of $F$, call them $F_0$ and $F_1$, glued along the curve $Q$.

\begin{lemma}\label{lem:admiss}
Consider $L_n'$ embedded in $F_0$, and push $L_n'$ slightly into $N$.  Then $\Sigma$ is an admissible surface for $L_n'$.
\end{lemma}
\begin{proof}
First, $F \X I$ is a genus four handlebody, as is $N$, since $N$ is obtained by crushing the vertical boundary of $F \X I$.  Moreover, since the exterior $E(Q)$ is fibered with fiber $F$, we may view this fibering as an open book decomposition of $S^3$ with binding $Q$, and thus $\overline{S^3 \setminus N}$ is homeomorphic to $N$, so that $\Sigma$ is a Heegaard surface for $S^3$.

It remains to be seen that there is a core of $N$ containing $L_n'$, but it suffices to show that there is a pair $D_n$ and $D_n'$ of dualizing disks for $L_n'$ in $N$.  Note that for any properly embedded arc $a \subset F_0$, there is a compressing disk $D(a)$ for $N$ obtained by crushing the vertical boundary of the disk $a \X I \subset F \X I$.  Let $a_0$ and $a_0'$ be disjoint arcs embedded in $F_0$ such that $a_0$ meets $V_n$ once and avoids $V_n'$, and $a_0'$ meets $V_n'$ once and avoids $V_n$.  Then $D(a_0)$ and $D(a_0')$ are dualizing disks for $L_n'$, completing the proof.
\end{proof}

Lemma~\ref{lem:admiss} does more than simply prove $\Sigma$ is admissible; it provides the key ingredients we need to construct a diagram for $\Tt(L_n',\Sigma)$:  Let $a_1$ and $a_1'$ denote parallel copies of $a_0$ and $a_0'$, respectively, in $F_1$, so that $\pd D(a_0) = a_0 \cup a_1$ and $\pd D(a_0') = a_0' \cup a_1'$.  By Lemma~\ref{lem:constr}, there is a genus four trisection diagram $(\A,\n,\g)$ for $\Tt(L_n',\Sigma)$ so that
\[ \n_1 = \pd D(a_0) \qquad \n_2 = \pd D(a_0') \qquad  \g_1 = V_n \qquad \g_2 = V_n'. \]
Noting that $(\n,\g)$ defines a genus four splitting of $Y_2$, it follows that any curve disjoint from $\n_1 \cup \n_2 \cup \g_1 \cup \g_2$ that bounds a disk in either of $H_{\n}$ or $H_{\g}$ also bounds in the other handlebody.  Let $b_0$ and $b_0'$ denote non-isotopic disjoint arcs in $F_0$ that are disjoint from $a_0 \cup a_0' \cup L_n'$.  Then $b_0 \cup b_1$ and $b_0' \cup b_1'$ bound disks in $N$; thus letting
\[\n_3 = \g_3 = b_0 \cup b_1 \qquad \n_4 = \g _4 = b_0' \cup b_1',\]
we have that $(\n,\g)$ is a standard diagram, corresponding to two of the cut systems in a diagram for $\Tt(L_n,\Sigma)$.  To find the curves in $\A$, let $N' = \overline{S^3 \setminus N}$, and observe that $N'$ also has the structure of $F \X I$ crushed along its vertical boundary, and $\pd N' = \pd N = F_0 \cup F_1$.

\begin{figure}[h!]
	\centering
	\includegraphics[width=.9\textwidth]{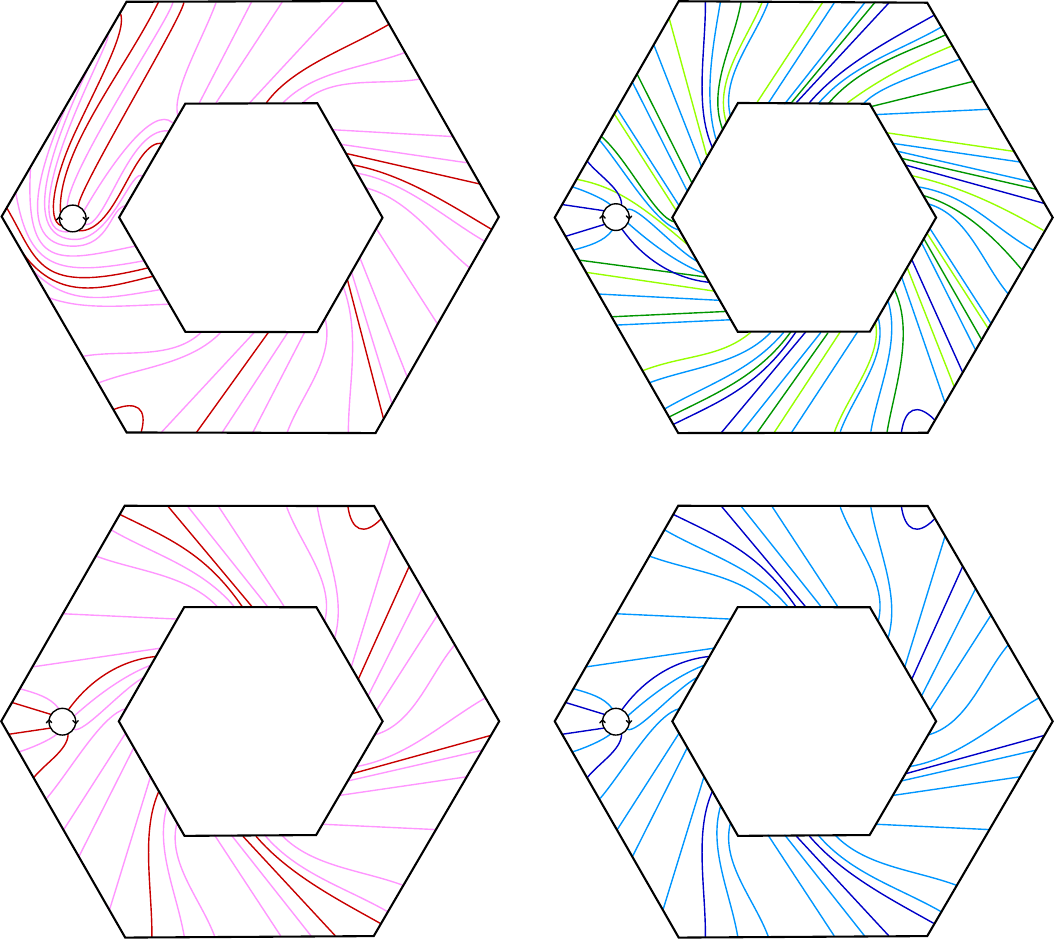}
	\caption{A trisection diagram for $\Tt(L_3',\Sigma)$.  The top row shows two copies of $F_0$, along with arcs: $\varphi(a_0)$ and $\varphi(a_0')$ (red), $\varphi(b_0)$ and $\varphi(b_0')$ (pink); $b_0$ and $b_0'$ (dark blue), $b_1$ and $b_1'$ (light blue), and $V_3$ (dark green) and $V_3'$ (light green).  The bottom row shows two copies of $F_1$, along with arcs: $a_1$ and $a_1'$ (red and dark blue) and $b_1$ and $b_1'$ (pink and light blue). The surfaces in the top row are identified with those in the bottom row along the oriented puncture.  Thus, each column describes the closed genus four surface $\Sigma$.  The left column encodes a 4--tuples of curves on this surface, namely, $\alpha$. The right column encodes the 4--tuple $\beta$ (shades of blue), as well as the two curves $\gamma_1$ and $\gamma_2$.  The trisection diagram for $\Tt(L_3',\Sigma)$ is obtained by overlaying the two columns. (Note that $\gamma_3 = \beta_3$ and $\gamma_4 = \beta_4$.)}
	\label{fig:GST}
\end{figure}

One way to reconstruct $S^3$ from $N$ and $N'$, both of which are homeomorphic to crushed products $F \X I$, is to initially glue $F_1 \subset \pd N'$ to $F_1 \subset \pd N'$.  The result of this initial gluing is again homeomorphic to a crushed product $F \X I$.  The second gluing then incorporates the monodromy, so that $F_0 \subset N'$ is glued to $F_0 \subset N$ via $\varphi$.  The result of this gluing is that if $a_1$ is an arc in $F_1 \subset N'$ and $D'(a_1)$ is the corresponding product disk in $N'$, then $\pd D'(a) = a_1 \cup \varphi(a_0)$, where $a_0$ is a parallel copy of $a_1$ in $F_0$ (using the product structure of $N$).

Thus, in order to find curves in $\A$, we can choose any four arcs in $F_1$ cutting the surface into a planar component and construct their product disks.  However, if we wish to a find a diagram with relatively little complication with respect to the $\n$ and $\g$ curves we have already chosen, it makes sense to choose those four arcs to be $a_1$, $a_1'$, $b_1$, and $b_1'$.  Thus,
\[ \A_1 = a_1 \cup \varphi(a_0) \qquad \A_2 = a_1' \cup \varphi(a_0') \qquad \A_3 = b_1 \cup \varphi(b_0) \qquad \A_4 = b_1' \cup \varphi(b_0').\]

We have proved the following:

\begin{proposition}
The triple $(\A,\n,\g)$ forms a $(4;0,2,2)$-trisection diagram for $\Tt(L_n,\Sigma)$.
\end{proposition}

%%%%%%%%%%%%%%%%%%%%%%%%%%%%%%%%%%%%%%%%%%%%%%%%%%%%%%%%%%%%%%%%%%%%%%%%
%%%%%%%%%%%%%%%%%%%%%%%%%%%%%%%%%%%%%%%%%%%%%%%%%%%%%%%%%%%%%%%%%%%%%%%%
\section{A rectangle condition for trisection diagrams}\label{sec:rect}
%%%%%%%%%%%%%%%%%%%%%%%%%%%%%%%%%%%%%%%%%%%%%%%%%%%%%%%%%%%%%%%%%%%%%%%%
%%%%%%%%%%%%%%%%%%%%%%%%%%%%%%%%%%%%%%%%%%%%%%%%%%%%%%%%%%%%%%%%%%%%%%%%

In the final section, we introduce a tool for potential future use.  This tool is an adaptation of the Rectangle Condition for Heegaard diagrams of Casson and Gordon~\cite{Casson-Gordon_Reducing_1987} to the setting of trisection diagrams.  A collection of $3g-3$ pairwise disjoint and nonisotopic curves in a genus $g$ surface $\Sigma$ is called a \emph{pants decomposition}, as the curves cut $\Sigma$ into $2g-2$ thrice-punctured spheres, or pairs of pants.  A pants decomposition defines a handlebody in the same way a cut system does, although a cut system is a minimal collection of curves defining a handlebody, whereas a pants decomposition necessarily contains superfluous curves.  An \emph{extended Heegaard diagram} is a pair of pants decompositions $(\Aa,\Nn)$ determining a Heegaard splitting $H_{\Aa} \cup H_{\Nn}$.  An \emph{extended trisection diagram} is a triple of pants decompositions $(\Aa,\Nn,\Gg)$ determining the spine $H_{\Aa} \cup H_{\Nn} \cup H_{\Gg}$ of a trisection.

Suppose that $\Aa$ and $\Nn$ are pants decompositions of $\Sigma$, and let $P_{\Aa}$ be a component of $\Sigma \setminus \nu(\Aa)$ and $P_{\Nn}$ a component of $\Sigma \setminus \nu(\Nn)$.  Let $a_1$, $a_2$, and $a_3$ denote the boundary components of $P_{\Aa}$; $b_1$, $b_2$, and $b_3$ the boundary components of $P_{\Nn}$.  We say that the pair $(P_{\Aa},P_{\Nn})$ is \emph{saturated} if for all $i,j,k,l \in \{1,2,3\}$, $i\neq j$, $k \neq l$, the intersection $P_{\Aa} \cap P_{\Nn}$ contains a rectangle $R_{i,j,k,l}$ with boundary arcs contained in $a_i$, $b_k$, $a_j,$ and $b_l$.  We say that that pair of pants $P_{\Aa}$ is \emph{saturated with respect to $\Nn$} if for every component $P_{\Nn}$ of $\Sigma \setminus \nu(\Nn)$, the pair $(P_{\Aa},P_{\Nn})$ is saturated.

An extended Heegaard diagram $(\Aa,\Nn)$ satisfies the Rectangle Condition of Casson-Gordon if for every component $P_{\Aa}$ of $\Sigma \setminus \nu(\Aa)$, we have $P_{\Aa}$ is saturated with respect to $\Nn$.  Casson and Gordon proved the following.

\begin{theorem}[\cite{Casson-Gordon_Reducing_1987}]\label{thm:CG}
Suppose that an extended Heegaard diagram $(\Aa,\Nn)$ satisfies the Rectangle Condition.  Then the induced Heegaard splitting $H_{\Aa} \cup H_{\Nn}$ is irreducible.
\end{theorem}

Now, let $(\Aa,\Nn,\Gg)$ be an extended trisection diagram.  We say that $(\Aa,\Nn,\Gg)$ satisfies the \emph{Rectangle Condition} if for every component $P_{\Aa}$ of $\Sigma \setminus \nu(\Aa)$, we have that either $P_{\Aa}$ is saturated with respect to $\Nn$ or $P_{\Aa}$ is saturated with respect to $\Gg$.

\begin{proposition}
Suppose that an extended trisection diagram satisfies the Rectangle Condition.  Then the induced trisection $\Tt$ with spine $H_{\Aa} \cup H_{\Nn} \cup H_{\Gg}$ is irreducible.
\end{proposition}
\begin{proof}
Suppose by way of contradiction that $\Tt$ is reducible.  Then there exists a curve $\delta \subset \Sigma = \pd H_{\Aa}$ that bounds disks $D_1 \subset H_{\Aa}$, $D_2 \subset H_{\Nn}$, and $D_3 \subset H_{\Gg}$.  Let $D_{\Aa}$ denote the set of $3g-3$ disks in $H_{\Aa}$ bounded by the curves $\Aa$, and define $D_{\Nn}$ and $D_{\Gg}$ similarly.  There are several cases to consider.  First, suppose that $\delta \in \Aa$, so that $D_1 \in D_{\Aa}$, and let $P_{\Aa}$ be a component of $\Sigma \setminus \nu(\Aa)$ that contains $\delta$ as a boundary component.  Suppose without loss of generality that $P_{\Aa}$ is saturated with respect to $\Nn$.  Then, for any curve $b \in \Nn$, we have that $b$ is the boundary of a component $P_{\Nn}$ of $\Sigma \setminus \nu(\Nn)$, where $P_{\Aa} \cap P_{\Nn}$ contains a rectangle with boundary arcs in $\delta$ and $b$.  It follows that $\delta$ meets every curve $b \in \Nn$, so $\delta \notin \Nn$.

Suppose that $D_2$ and $D_{\Nn}$ have been isotoped to intersect minimally, so that these disks meet in arcs by a standard argument.  There must be an outermost arc of intersection in $D_2$, which bounds a subdisk of $D_2$ with an arc $\delta' \subset \delta$, and $\delta'$ is a \emph{wave} (an arc with both endpoints on the same boundary curve) contained a single component $P_{\Nn}$ of $\Sigma \setminus \nu(\Nn)$.  Let $b_1$ and $b_2$ be the boundary components of $P_{\Nn}$ disjoint from $\delta'$.  Since $P_{\Aa}$ is saturated with respect to $\Nn$, there is rectangle $R \subset P_{\Aa} \cap P_{\Nn}$ with boundary arcs contained in $b_1$, $\delta$, $b_2$, and some other curve in $\pd P_{\Aa}$.  Let $\delta''$ be the arc component of $\pd R$ contained in $\delta$.  Since the wave $\delta'$ separates $b_1$ from $b_2$ in $P_{\Nn}$, it follows that $\delta' \cap \delta'' \neq \emp$, a contradiction.

In the second case, suppose that $\delta$ is a curve in $\Nn$.  Note that the Heegaard splitting determined by $(\Aa,\Gg)$ is reducible, and thus by the contrapositive of Casson-Gordon's rectangle condition, there must be some pants decomposition $P_{\Aa}$ of $\Sigma \setminus \nu(\Aa)$ such that $P_{\Aa}$ is \textbf{not} saturated with respect to $\Gg$, so that $P_{\Aa}$ is saturated with respect to $\beta$.  Let $P_{\Nn}$ be a component of $\Sigma \setminus \nu(\Nn)$ that contains $\delta$ as boundary component.  By the above argument, $\delta \notin \Aa$, and if we intersect $D_1$ with $D_{\Aa}$ we can run an argument parallel to the one above to show that $\delta$ has a self-intersection, a contradiction.  A parallel argument shows that $\delta \notin \Gg$.

Finally, suppose that $\delta$ is not contained in any of $\Aa$, $\Nn$, or $\Gg$.  By intersecting the disks $D_1$ and $D_{\Aa}$, we see that there is a wave $\delta' \subset \delta$ contained in some pants component $P_{\Aa}$ of $\Sigma \setminus \nu(\Aa)$.  Suppose without loss of generality that $P_{\Aa}$ is saturated with respect to $\Nn$.  By intersecting $D_2$ with $D_{\Nn}$, we see that there is a wave $\delta'' \subset \delta$ contained in some pants component $P_{\Nn}$ of $\Sigma \setminus \nu(\Nn)$.  Let $a_1$ and $a_2$ be the components of $\pd P_{\Aa}$ that avoid $\delta'$, and let $b_1$ and $b_2$ be the components of $\pd P_{\Nn}$ that avoid $\delta''$.  By the Rectangle Condition, $P_{\Aa} \cap P_{\Nn}$ contains a rectangle $R$ whose boundary is made of arcs in $a_1$, $b_1$, $a_2$, and $b_2$.   As such, $\delta' \cap R$ contains an arc connecting $b_1$ to $b_2$, while $\delta'' \cap R$ contains an arc connecting $a_1$ to $a_2$, but this implies that $\delta' \cap \delta'' \neq \emp$, a contradiction.  We conclude that no such curve $\delta$ exists.
\end{proof}

Of course, at this time, the Rectangle Condition is a tool without an application, which elicits the following question:

\begin{question}
Is there an extended trisection diagram $(\Aa,\Nn,\Gg)$ that satisfies the Rectangle Condition?
\end{question}

Note that while it is easy to find three pants decompositions that satisfy the Rectangle Condition, the difficulty lies in finding three such pants decompositions which also determine a trisection; in pairs, they must be extended Heegaard diagrams for the 3-manifolds $Y_k$.

\bibliographystyle{amsalpha}
\bibliography{MasterBibliography_2017_06}

\providecommand{\bysame}{\leavevmode\hbox to3em{\hrulefill}\thinspace}
\providecommand{\MR}{\relax\ifhmode\unskip\space\fi MR }
\providecommand{\MRhref}[2]{%
  \href{http://www.ams.org/mathscinet-getitem?mr=#1}{#2}
}
\providecommand{\href}[2]{#2}
\begin{thebibliography}{MSZ16}

\bibitem[AK85]{Akbulut-Kirby_A-potential_1985}
Selman Akbulut and Robion Kirby, \emph{A potential smooth counterexample in
  dimension {$4$} to the {P}oincar{\'e} conjecture, the {S}choenflies
  conjecture, and the {A}ndrews-{C}urtis conjecture}, Topology \textbf{24}
  (1985), no.~4, 375--390. \MR{816520}

\bibitem[CG87]{Casson-Gordon_Reducing_1987}
A.~J. Casson and C.~McA. Gordon, \emph{Reducing {H}eegaard splittings},
  Topology Appl. \textbf{27} (1987), no.~3, 275--283. \MR{918537 (89c:57020)}

\bibitem[Gab87]{Gabai_FoliationsIII_1987}
David Gabai, \emph{Foliations and the topology of {$3$}-manifolds. {III}}, J.
  Differential Geom. \textbf{26} (1987), no.~3, 479--536. \MR{910018
  (89a:57014b)}

\bibitem[GK16]{Gay-Kirby_Trisecting_2016}
David Gay and Robion Kirby, \emph{Trisecting 4-manifolds}, Geom. Topol.
  \textbf{20} (2016), no.~6, 3097--3132. \MR{3590351}

\bibitem[GL89]{Gordon-Luecke_Complements_1989}
C.~McA. Gordon and J.~Luecke, \emph{Knots are determined by their complements},
  Bull. Amer. Math. Soc. (N.S.) \textbf{20} (1989), no.~1, 83--87. \MR{972070
  (90a:57006b)}

\bibitem[GST10]{Gompf-Scharlemann-Thompson_Fibered_2010}
Robert~E. Gompf, Martin Scharlemann, and Abigail Thompson, \emph{Fibered knots
  and potential counterexamples to the property 2{R} and slice-ribbon
  conjectures}, Geom. Topol. \textbf{14} (2010), no.~4, 2305--2347. \MR{2740649
  (2012c:57012)}

\bibitem[Hak68]{Haken_Some_1968}
Wolfgang Haken, \emph{Some results on surfaces in {$3$}-manifolds}, Studies in
  {M}odern {T}opology, Math. Assoc. Amer. (distributed by Prentice-Hall,
  Englewood Cliffs, N.J.), 1968, pp.~39--98. \MR{0224071 (36 \#7118)}

\bibitem[Kir78a]{Kirby_Problems_1978}
Rob Kirby, \emph{Problems in low dimensional manifold theory}, Algebraic and
  geometric topology ({P}roc. {S}ympos. {P}ure {M}ath., {S}tanford {U}niv.,
  {S}tanford, {C}alif., 1976), {P}art 2, Proc. Sympos. Pure Math., XXXII, Amer.
  Math. Soc., Providence, R.I., 1978, pp.~273--312. \MR{520548 (80g:57002)}

\bibitem[Kir78b]{Kirby_A-calculus_1978}
Robion Kirby, \emph{A calculus for framed links in {$S^{3}$}}, Invent. Math.
  \textbf{45} (1978), no.~1, 35--56. \MR{0467753}

\bibitem[Lic62]{Lickorish_A-representation_1962}
W.~B.~R. Lickorish, \emph{A representation of orientable combinatorial
  {$3$}-manifolds}, Ann. of Math. (2) \textbf{76} (1962), 531--540. \MR{0151948
  (27 \#1929)}

\bibitem[LP72]{Laudenbach-Poenaru_A-note_1972}
Fran{\c{c}}ois Laudenbach and Valentin Po{{\'e}}naru, \emph{A note on
  {$4$}-dimensional handlebodies}, Bull. Soc. Math. France \textbf{100} (1972),
  337--344. \MR{0317343 (47 \#5890)}

\bibitem[Mon79]{Montesinos_Heegaard_1979}
Jos{\'e} Mar\'\i~a Montesinos, \emph{Heegaard diagrams for closed
  {$4$}-manifolds}, Geometric topology ({P}roc. {G}eorgia {T}opology {C}onf.,
  {A}thens, {G}a., 1977), Academic Press, New York-London, 1979, pp.~219--237.
  \MR{537732}

\bibitem[MSZ16]{Meier-Schirmer-Zupan_Classification_2016}
Jeffrey Meier, Trent Schirmer, and Alexander Zupan, \emph{Classification of
  trisections and the {G}eneralized {P}roperty {R} {C}onjecture}, Proc. Amer.
  Math. Soc. \textbf{144} (2016), no.~11, 4983--4997. \MR{3544545}

\bibitem[MZ]{Meier-Zupan_Fibered_}
Jeffrey Meier and Alexander Zupan, \emph{Fibered homotopy-ribbon knots, the
  {G}eneralized {P}roperty {R} {C}onjecture, and trisections}, \emph{In
  preparation}.

\bibitem[MZ17]{Meier-Zupan_Genus-two_2017}
\bysame, \emph{Genus-two trisections are standard}, Geom. Topol. \textbf{21}
  (2017), no.~3, 1583--1630. \MR{3650079}

\bibitem[OS06]{Ozsvath-Szabo_Lectures_2006}
Peter Ozsv{{\'a}}th and Zolt{{\'a}}n Szab{{\'o}}, \emph{Lectures on {H}eegaard
  {F}loer homology}, Floer homology, gauge theory, and low-dimensional
  topology, Clay Math. Proc., vol.~5, Amer. Math. Soc., Providence, RI, 2006,
  pp.~29--70. \MR{2249248 (2007g:57053)}

\bibitem[Rei33]{Reidemeister_Zur-dreidimensionalen_1933}
Kurt Reidemeister, \emph{Zur dreidimensionalen {T}opologie}, Abh. Math. Sem.
  Univ. Hamburg \textbf{9} (1933), no.~1, 189--194. \MR{3069596}

\bibitem[Sch12]{Scharlemann_Proposed_2012}
Martin Scharlemann, \emph{Proposed {P}roperty 2{R} counterexamples classified},
  arXiv:1208.1299.

\bibitem[Sin33]{Singer_Three-dimensional_1933}
James Singer, \emph{Three-dimensional manifolds and their {H}eegaard diagrams},
  Trans. Amer. Math. Soc. \textbf{35} (1933), no.~1, 88--111. \MR{1501673}

\bibitem[Wal60]{Wallace_Modifications_1960}
Andrew~H. Wallace, \emph{Modifications and cobounding manifolds}, Canad. J.
  Math. \textbf{12} (1960), 503--528. \MR{0125588 (23 \#A2887)}

\bibitem[Wal68]{Waldhausen_Heegaard-Zerlegungen_1968}
Friedhelm Waldhausen, \emph{Heegaard-{Z}erlegungen der {$3$}-{S}ph{\"a}re},
  Topology \textbf{7} (1968), 195--203. \MR{0227992 (37 \#3576)}

\end{thebibliography}

\end{document}